\documentclass[12pt]{amsart}
\usepackage{amscd,amsmath,amssymb,amsfonts}
\usepackage[cmtip, all]{xy}
\usepackage{enumerate}

\theoremstyle{plain}
\newtheorem{thm}[equation]{Theorem}

\newtheorem{lem}[equation]{Lemma}
\newtheorem{cor}[equation]{Corollary}

\newtheorem{prop}[equation]{Proposition}

\newtheorem{hyp}[equation]{Hypothesis}

\theoremstyle{definition}

\newtheorem{disc}[equation]{Comment}
\newtheorem{question}[equation]{Question}

\numberwithin{equation}{section}

\newcommand{\isoarrow}{{~\overset\sim\longrightarrow}}

\bigskip

\newcommand{\CO}{{\mathcal{O}}}

\newcommand{\ZZ}{{\mathbb Z}}

\newcommand{\ra}{{~\rightarrow~}}

\newcommand{\QQ}{{\mathbb Q}}

\newcommand{\Ql}{{\mathbb Q}_{\ell}}

\newcommand{\RR}{{\mathbb R}}

\newcommand{\af}{{{\bold A}_f}}

\newcommand{\CC}{{\mathbb C}}

\newcommand{\Qp}{{{\mathbb Q}_p}}

\begin{document}

\title [Weight zero Eisenstein cohomology via Berkovich spaces]
{Weight zero Eisenstein cohomology of Shimura varieties via Berkovich spaces}

\author{Michael Harris}
\address{Institut de Math\'ematiques de
Jussieu, U.M.R. 7586 du CNRS; UFR de Math\'ematiques \\ Universit\'e Paris-Diderot Paris 7}

\thanks{The research leading to these results has received funding from the European Research Council under the European Community's Seventh Framework Programme (FP7/2007-2013) / ERC Grant agreement n¡ 290766 (AAMOT)}

\dedicatory{in memory of Jon Rogawski}
\smallskip

\subjclass[2000]{}

\date{\today}
\thanks{ }

\maketitle

\section*{Introduction}

This paper represents a first attempt to understand a geometric structure
that plays an essential role in my forthcoming paper with Lan, Taylor, and Thorne \cite{HLTT}
on the construction of certain Galois representations by $p$-adic interpolation
between Eisenstein cohomology classes and cuspidal cohomology.  The classes
arise from the cohomology of a locally symmetric space $Z$ without complex structure --
specifically, the adelic locally symmetric space attached to $GL(n)$ over a CM field $F$.
It has long been known, thanks especially to the work of Harder and Schwermer
(cf. \cite{H}, for example) that classes of this type often give rise to non-trivial Eisenstein
cohomology of a Shimura variety $S$; in the case of  $GL(n)$ as above, $S$ is attached to
the unitary similitude group of a maximally isotropic hermitian space of dimension $2n$
over $F$.   This is the starting point of the connection with
Galois representations.  The complete history of this idea will be explained in \cite{HLTT}; in this
note I just want to explore a different perspective on the construction of these classes.

By duality, the Eisenstein classes of Harder and Schwermer correspond to classes 
in cohomology with compact support, and it turns out to be more fruitful to look
at them in this way.  One of Taylor's crucial observations was that certain of these
classes are of weight zero and can therefore be constructed geometrically in any
cohomology theory with a good weight filtration, in particular in rigid cohomology, which
lends itself to $p$-adic interpolation.   The geometric construction involves the abstract simplicial
complex $\Sigma$ defined by the configuration of boundary divisors of a toroidal compactification;
this complex, which is homotopy equivalent to the original locally symmetric space $Z$,
 arises in the calculation of the weight filtration on the cohomology of the
logarithmic de Rham complex.  Under certain conditions (as I was reminded by
Wiesia Nizio\l, commenting on the construction in \cite{HLTT}\footnote{Since writing the first
version of this article I have learned that Laurent Fargues had essentially the same idea independently.}) 
Berkovich has defined an 
isomorphism between the weight zero cohomology with compact support of a scheme and
the compactly supported cohomology of the associated Berkovich analytic space, which 
is a topological space.  His results apply to both $\ell$-adic and $p$-adic \'etale cohomology
as well as to Hodge theory.  In this paper we apply this isomorphism to the toroidal compactification $S'$
of a Shimura variety $S$.    Both $S$ and $S'$ are defined over some number field $E$; we fix
a place $v$ of $E$ dividing the rational prime $p$ and let $|S|$ and $|S'|$ be the associated analytic spaces over $E_v$ 
in the sense of Berkovich.   We observe that $\Sigma$ is homotopy equivalent to  $|S'| \setminus |S|$
\footnote{Our Shimura varieties are attached to groups of rational rank $1$, whose toroidal boundary 
is the blowup of point boundary components in the minimal compactification.  More general Shimura
varieties are compactified by adding strata attached to different conjugacy classes of maximal
parabolics, and then $|S'| \setminus |S|$ has several strata as well.}   Moreover, when $S$ 
and $S'$ both have good reduction at $v$, $|S|$ and $|S'|$ 
are both contractible \cite{B1}, and it follows easily that the cohomology of $\Sigma$ maps to $H^*_c(|S'|)$ in
the theories considered in \cite{B2}.  

These ideas will be worked out systematically in forthcoming work.
The present note explains the construction in the simplest situation.  We only consider
cohomology with trivial coefficients of Shimura varieties with a single class of rational boundary
components, assumed to be of dimension $0$.   We work with connected rather than
adelic Shimura varieties and write the boundary as a union of connected quotients
of a (non-hermitian) symmetric space by discrete subgroups.  We also only work at places
of good reduction, in order to quote Berkovich's theorems directly.  In \cite{HLTT} it is crucial to 
at arbitrary level, but the relevant target spaces are the ordinary loci of Shimura varieties.
Perhaps Berkovich's methods apply to these spaces as well, but for the moment this cannot
be used, since the results of \cite{B2} have not been verified for rigid cohomology.

Berkovich gives a topological interpretation in \cite{B2} of the weight zero stage of the Hodge
filtration, but it can also be used as a topological definition of this part of the cohomology.
Since the cohomology of $\Sigma$ has a natural integral structure, it's conceivable that
the results of Berkovich provide some information about torsion in cohomology.  This is one of
the main motivations for reconsidering the construction of \cite{HLTT} in the light of Berkovich's theory.

I also thank my coauthors Kai-Wen Lan, Richard 
Taylor, and Jack Thorne, for providing the occasion for the present paper;
 Wiesia Nizio{\l}, for pointing out the connection with Berkovich's work; and Sam Payne, for
explaining the results of Berkovich and Thuillier.   

Jon Rogawski was exceptionally generous in person.
Although we met rarely,  on at least two separate occasions he took the time
to help me find my way around technical problems central to the
success of my work.  He will be greatly missed.

\section{The construction}

The standard terminology and notation for Shimura varieties will be used without
explanation.   Let $(G,X)$ be the datum defining a Shimura variety $S(G,X)$, with $G$ a connected reductive
group over $\QQ$ and $X$ a union of copies of the hermitian symmetric space attached to
the identity component of $G(\RR)$.  Let $D$ be one of these components and let $\Gamma \subset G(\QQ)$
be a congruence subgroup; then $S = \Gamma \backslash D$ is a connected component of $S(G,X)$ at
some finite level $K$; here $K$ is an open compact subgroup of $G(\af)$.  Then $S$ has a canonical model over some number field $E = E(D,\Gamma)$.  We assume $\Gamma$ is neat; then $S$ is smooth and has a family
of smooth projective toroidal compactifications, as in \cite{AMRT}.   We make a series of simplifying hypotheses.

\begin{hyp}\label{rank}  The group $G$ has rational rank $1$.  Let $P$ be a rational parabolic subgroup of $G$ (unique
up to conjugacy); then $P$ is the stabilizer of a point boundary component of $D$.
\end{hyp}

It follows from the general theory in \cite{AMRT}
that if  $S' \supset S$ is a toroidal compactification then the complement $S' \setminus S$ is a
union of rational divisors.  We pick such an $S'$, assumed smooth and projective, and assume that 
$S' \setminus S$ is a divisor with normal crossings.  Let $v$ be
a place of $E$ dividing the rational prime $p$. 

\begin{hyp}\label{gr}  The varieties $S$ and $S'$ have smooth projective models $\frak{S}$ and $\frak{S'}$
over the $v$-adic integer ring $Spec(\CO_v)$.
\end{hyp}

This is proved by Lan for PEL type Shimura varieties at hyperspecial level in several long papers,
starting with \cite{L}.  

Let $\bar{\frak{S}}$ denote the base change of $\frak{S}$ to the algebraic closure of $F_v$,
$\bar{\frak{S}}^{an}$ the associated Berkovich analytic space, and $|S|$ (rather than
$|\bar{\frak{S}}^{an}|$) the underlying topological space.  We use the same notation for $S'$.  
Let $Z = |S'| \setminus |S|$.

\begin{lem}  The spaces $|S|$ and $|S'|$ are contractible and $|S'|$ is compact.  In particular, 
\begin{itemize}
\item[(a)]  The inclusion $|S| \hookrightarrow |S'|$ is a homotopy equivalence;
\item[(b)]  There are canonical isomorphisms $H^i_c(|S|,A) \isoarrow H^i(|S'|,Z;A)$
for any ring $A$;
\item[(c)]  The connecting homomorphism $H^i(Z,A) \rightarrow H^{i+1}(|S'|,Z;A) = H^{i+1}_c(|S'|,A) $ is an isomorphism
for $i > 0$.
\end{itemize}
\end{lem}

\begin{proof}  Contractibility of $|S|$ and $|S'|$ follows from \ref{gr} by the results of \S 5 of \cite{B1} (though the contractibility
of analytifications of spaces with good reduction seems only to be stated explicitly in the introduction).  
Since $S'$ is proper, $|S'|$ is compact.  Then (a) and (b) are clear and (c) follows from the long exact sequence for
cohomology
$$\dots H^i(|S'|,A) \ra H^i(Z,A) \rightarrow H^{i+1}(|S'|,Z;A) \ra H^{i+1}(|S'|,A) \dots $$
\end{proof}

Let $P = LU$ be a Levi decomposition, with $L$ reductive and $U$
unipotent, let $L^0$ denote the identity component of the Lie group $L(\RR)$,
and let $D_P$ denote the symmetric space attached to $L^0$ (or to its derived subgroup $(L^0)^{der}$).
The {\it minimal compactification} (or Satake compactification) $S^*$ of $S$ is a projective algebraic
variety obtained by adding a finite set of points, say $N$ points, which we can call ``cusps," to $S$.  The {\it toroidal compactifications} $S^{tor}$
of \cite{AMRT}, which depend on combinatorial data, are constructed by blowing up the cusps;  each one is replaced by a configuration of rational divisors, to which we return momentarily.  For appropriate choices of data $S^{tor}$ is a smooth
projective variety and $\partial S^{tor} = S^{tor} \setminus S$ is a divisor with normal crossings; $\partial S^{tor}$ is a
union of $N$ connected components, one for each cusp.
The {\it reductive Borel-Serre compactification} $S^{rs}$ of $S$ is a compact (non-algebraic)
manifold with corners (boundary in this case) containing $S$ as dense open subset, and such that
\begin{equation}  S^{rs} \setminus S = \coprod_{j = 1}^N \Delta_j\backslash D_P \end{equation}
where for each $j$ $\Delta_j$ is a cocompact congruence subgroup of $L(\QQ)$.  

Details on $S^{rs}$ can be found in a number of places, for example  \cite{BJ}.  We introduce
this space only in order to provide an independent description of $Z$.  Roughly speaking, $Z$ is canonically
homotopy equivalent to $S^{rs} \setminus S$.   More precisely, let $S \hookrightarrow S^{tor}$ be
a toroidal compactification as above consider the {\it incidence complex} $\Sigma$ of the divisor with normal crossings
$\partial S^{tor}$.  This is a simplicial complex whose vertices are the irreducible components $\partial_i$ of $\partial S^{tor}$,
whose edges are the non-trivial intersections $\partial_i \cap \partial_j$, and so on.  

\begin{prop}\label{triang}  The incidence complex $\Sigma$ is homeomorphic to a triangulation 
of $S^{rs} \setminus S$.
\end{prop}
\begin{proof}  In \cite{HZ}, Corollary 2.2.10, it is proved that $\Sigma$ is a triangulation
of a compact deformation retract of $S^{rs} \setminus S$; but under \ref{rank} $S^{rs} \setminus S$
is already compact.  In any case, $\Sigma$ and $S^{rs} \setminus S$ are homotopy equivalent.
\end{proof}

\begin{thm}\label{defret} (Berkovich and Thuillier)  There is a canonical deformation retraction of $Z = S'  \setminus S$
onto $\Sigma$.
\end{thm}

This is proved but not stated in \cite{T}, and can also be extracted from \cite{B1}.  A more precise reference
will be provided in the sequel.  

In what follows, $H^{\bullet}_c$ will be one of the cohomology theories (a') $H^{\bullet}_{\ell,c}$, 
(a'') $H^{\bullet}_{p,c}$ ($\ell$-adic or $p$-adic \'etale cohomology, with $\ell \neq p$) or 
(c) $V \mapsto H^{\bullet}_c(V(\CC), \QQ)$
(Betti cohomology with compact support of the complex points of the algebraic variety $V$), considered in Theorem 1.1
of \cite{B2}.   Corresponding to the choice of $H^{\bullet}_c$, the ring $A$ is either (a') $\Ql$, (a'')$\Qp$, or (c) $\QQ$.

\begin{cor}\label{inj}  For $i > 0$, there is a canonical injection
$$  \phi:  H^i(S^{rs} \setminus S,A) = H^i(\coprod_{j = 1}^N \Delta_j\backslash D_P,A) \hookrightarrow H^{i+1}_c(\bar{S}).$$
The image of $\phi$ is the weight zero subspace in cases (a') and (c) and is the space of
smooth vectors for the action of the Galois group (see \cite{B2}, p. 666 for the definition) in case (a'').
\end{cor}
\begin{proof}  This follows directly from \ref{triang}, \ref{defret} and Theorem 1.1 of \cite{B2}.
\end{proof}

The key word is {\it canonical}. This means that the retractions commute with
change of discrete group $\Gamma$ (provided the condition \ref{gr} is preserved) and with Hecke correspondences,
or (more usefully) the action of the group $G(\af)$ in the adelic Shimura varieties.   In particular, the adelic version of
the corollary asserts roughly that the induced representation from $P(\af)$ to $G(\af)$ of the topological cohomology
of the locally symmetric space attached to $L$ injects into the cohomology with compact support
of the adelic Shimura variety $S(G,X)$, with image by either the weight zero subspace or the smooth vectors for
the Galois action.

\section{Some extensions and questions}

\begin{disc}  Hypothesis \ref{rank} is superfluous.  The homotopy type of $|S'| \setminus |S|$ is more complicated
but can be described along lines similar to \ref{defret}.
\end{disc} 

\begin{disc} The article \cite{HLTT} treats more general local coefficients by studying the weight zero cohomology of Kuga
families of abelian varieties over Shimura varieties.  The analytic space of the boundary in this case is a torus bundle with
fiber $(S^1)^d$ over the base $Z$, where $d$ is the relative dimension of the Kuga family over $S$.  The Leray spectral
sequence identifies the cohomology of the total space as the cohomology of $Z$ with coefficients in a sum of local systems
attached to irreducible representations of $L$.  In this way one can recover the Eisenstein classes of \cite{HLTT} for
general coefficients.
\end{disc}

\begin{question}  In \cite{HZ} the combinatorial calculation of the boundary contribution to
coherent cohomology is accompanied by a differential calculation, in which the Dolbeault complex
near the toroidal boundary is compared to the de Rham complex on the incidence complex.
Does this have analogues in other cohomology theories?
\end{question}

\begin{question}  Is there a version of Berkovich's theorem in \cite{B2} for local systems that works directly with $Z$ and
$S$ and avoids the use of Kuga families?   For $\ell$ prime to $p$, local systems over $S$ with coefficients in $\ZZ/\ell^n\ZZ$, attached to algebraic representations of $G$, become trivial when $\Gamma$ is replaced by an appropriate subgroup
of finite index.  This suggests that the analogue of \ref{inj} for $\ell$-adic cohomology with twisted coefficients can be proved
directly on the adelic Shimura variety.  It's not so clear how to handle cases (a'') and (c).
\end{question}

\begin{question}  Does Berkovich's theorem apply to rigid cohomology, which is the theory used in \cite{HLTT}?  
In particular, does it apply to the ordinary locus of the toroidal compactification?
\end{question}

\begin{question}  Most importantly, is there a version of \ref{inj} that keeps track of the torsion cohomology of
$S^{rs} \setminus S$?  The possibility of assigning Galois representation to torsion cohomology classes is
the subject of a series of increasingly precise and increasingly influential conjectures.  Can the methods of \cite{HLTT} be adapted to account for these classes?
\end{question}


\begin{thebibliography}{99}

\bibitem [AMRT]{AMRT}  A. Ash, D. Mumford, M. Rapoport, Y.-S. Tai, , Smooth
compactification of Locally Symmetric Varieties, Math.~Sci.~Press,
Brookline, MA, (1975).
\medskip


\bibitem [B1]{B1}  V. Berkovich, Smooth p-adic analytic spaces are locally contractible, {\it Invent. Math.}, {\bf 137} (1999), 
1--84.
\medskip

\bibitem [B2]{B2}  V. Berkovich, An analog of Tate's conjecture over local and finitely generated
fields, {\it IMRN}, {\bf  13} (2000)  665--680.
\medskip


\bibitem [BJ]{BJ} A. Borel and L. Ji, Compactifications of Locally Symmetric Spaces
{\it J. Differential Geom.},  {\bf 73},  (2006), 263-317. 
\medskip


\bibitem [H]{H}  G. Harder, Some results on the Eisenstein cohomology of arithmetic subgroups of $GL_n$,
{\it Lecture Notes in Mathematics}, {\bf 1447} (1990), 85--153.
\medskip




\bibitem[HLTT]{HLTT}
 M. Harris, K.-W. Lan, R. Taylor, and J. Thorne, On the rigid cohomology of certain Shimura varieties, in preparation.
\medskip



\bibitem[HZ]{HZ}  M. Harris, S. Zucker,  Boundary cohomology of Shimura
varieties, I: coherent cohomology on the toroidal boundary,  {\it Ann. Scient.
Ec. Norm. Sup.}, {\bf 27}, 249--344 (1994).

\medskip

\bibitem[L]{L}  K.-W. Lan, {\it Arithmetic compactifications of PEL-type Shimura varieties}, {\it London Mathematical Society Monographs}, Princeton University Press, in press.

\medskip


\bibitem[P]{P} R. Pink, Arithmetical compactification of mixed Shimura varieties. {\it  Bonner Mathematische Schriften }, {\bf 209} (1990).
\medskip



\bibitem [T]{T}  A. Thuillier, G\'eom\'etrie toro\"idale et g\'eom\'etrie analytique non archim\'edienne. Application au type d'homotopie de certains sch\'emas formels, {\it Manuscripta Math.}, {\bf 123} (2007)  381-451.



\end{thebibliography}
\end{document}